\documentclass[12pt]{amsart}
\usepackage{amssymb,upref,enumerate}

\usepackage{amsmath,amssymb,amscd,amsthm,amsxtra}
\usepackage{latexsym}
\usepackage{graphics,epsfig}
\allowdisplaybreaks

\def\XXint#1#2#3{{\setbox0=\hbox{$#1{#2#3}{\int}$}
\vcenter{\hbox{$#2#3$}}\kern-.5\wd0}}

\newtheorem{theorem}{Theorem}
\newtheorem{proposition}{Proposition}

\newtheorem{lemma}{Lemma}

\newtheorem{corollary}{Corollary}

\theoremstyle{remark}

\def\({\left(}
\def\){\right)}
\def\be {\begin{equation}}
\def\en{\end{equation}}

\def\Cdot{{\dot{C}}}
\def\div{\text{div}~}

\newcommand{\tensor}{\otimes}

\numberwithin{equation}{section}

\begin{document}


\subjclass[2010]{Primary 76D05, 35A02; Secondary 35K58}

\keywords{Navier-Stokes, Lagrangian Averaging, global existence, Besov spaces}

\address{Nathan Pennington, Department of Mathematics, Kansas State University, 138 Cardwell Hall,
Manhattan, KS-66506, USA.} \email{npenning@math.ksu.edu}


\title[Global solutions to the LANS in $B^{n/p}_{p,q}(\mathbb{R}^n)$]{Global Solutions to the Lagrangian Averaged Navier-Stokes equation in low regularity Besov spaces}
\author{Nathan Pennington}
\date{\today}

\begin{abstract} The Lagrangian Averaged Navier-Stokes (LANS) equations are a recently derived approximation to the Navier-Stokes equations.  Existence of global solutions for the LANS equation has been proven for initial data in the Sobolev space $H^{3/4,2}(\mathbb{R}^3)$ and in the Besov space $B^{3/2}_{2,q}(\mathbb{R}^3)$.  In this paper, we use an interpolation based method to prove the existence of global solutions to the LANS equation with initial data in $B^{3/p}_{p,q}(\mathbb{R}^3)$ for any $p>3$.  

\end{abstract}

\maketitle

\section{Introduction and Main Results}

The LANS equation is a recently derived approximation to the Navier-Stokes equation and is derived by averaging at the Lagrangian level.  For an exhaustive treatment of this process, see \cite{Shkoller}, \cite{SK}, \cite{MRS} and \cite{MS2}. In \cite{MKSM} and \cite{CHMZ}, the authors discuss the numerical improvements that use of the LANS equation provides over more common approximation techniques of the Navier-Stokes equation.

On $\mathbb{R}^n$, the isotropic, incompressible form of the LANS equation is given by 
\begin{equation}\label{LANS}\aligned\partial_t w+(w\cdot\nabla)w+\div
\tau^\alpha (w,w)=-(1-\alpha^2\triangle)^{-1}\text{grad}~p+\nu\triangle
w
\\ w:[0,T)\times\mathbb{R}^n\rightarrow \mathbb{R}^n, w(0,x)=w_0(x), ~~ \div w=\div w_0=0,
\endaligned
\end{equation}
where all the differential operators (except $\partial_t$) are spatial differential operators, $\alpha>0$ is a constant, $\nu>0$ is the viscosity of the fluid, $p$ denotes the fluid pressure, and $w_0$ is the initial data. The Reynolds
stress $\tau^\alpha(w,w)$ is given by
\begin{equation}\label{tau}\tau^\alpha (f,g)=\frac{\alpha^2}{2}(1-\alpha^2\triangle)^{-1}[Def(f)\cdot
Rot(g)+Def(f)\cdot Rot(g)]
\end{equation}
where $Rot(f)=(\nabla f-\nabla f^T)/2$ and $Def(f)=(\nabla f+\nabla f^T)/2$.  Abusing notation, we set $\tau^\alpha(f,f)=\tau^\alpha(f)$.  We note that setting $\alpha=0$ returns the Navier-Stokes equation.

The difference between the LANS equation and the Navier-Stokes equation is the additional nonlinear term $\tau^\alpha$.  This additional term complicates local existence theory, but makes it easier to control the long time behavior of local solutions.  Local existence results for the LANS equation in various settings can be found in \cite{Shkoller}, \cite{MRS}, \cite{MS}, \cite{sobpaper} and \cite{besovpaper}.  In \cite{MS}, Marsden and Shkoller proved the existence of a global solution to the LANS equation with initial data in the Sobolev space $H^{3,2}(\mathbb{R}^3)$.  In \cite{sobpaper}, this result was improved, achieving global existence for data in the space $H^{3/4,2}(\mathbb{R}^3)$.  In \cite{besovpaper}, existence of local solutions was proven for initial data in Besov spaces, and the local solution is extended to a global solution for initial data in $B^{s}_{2,q}(\mathbb{R}^3)$ for $s>3/4$. 

In this article we prove new global existence results to the LANS equation, guided by the method used by Gallagher and Planchon in \cite{galplan} (which has its origins in \cite{calns}) for the Navier-Stokes equation, which will be outlined below.  We now state the main result of this article.
\begin{theorem}\label{main theorem}Let $w_0\in B^{3/p}_{p,q}(\mathbb{R}^3)$ be divergence free.  Then there exists a unique global solution to the LANS equation $w\in C([0,\infty):B^{3/p}_{p,q}(\mathbb{R}^3))$ with $w(0)=w_0$, provided $p>2$.    
\end{theorem}
This result expands on the global existence result from \cite{besovpaper}, which only held in the case $p=2$.  The primary emphasis here is the large $p$ case, where we obtain global existence for data with regularity close to zero.

The rest of this section is devoted to proving Theorem $\ref{main theorem}$, up to Theorem $\ref{mLANS main theorem}$, the proof of which is the focus of the rest of the article.  We start with our solution space $W=B^{3/p}_{p,q}(\mathbb{R}^3)$, and define $U=B^{3/2}_{2,q}(\mathbb{R}^3)$ and $V=B^{3/\tilde{p}}_{\tilde{p},q}(\mathbb{R}^3)$, where $\tilde{p}>>p$.  Then, choosing $\theta$ such that $3/p=3\theta/2+3(1-\theta)/\tilde{p}$, we have that  
\begin{equation}\label{interp}W=[U,V]_{\theta,q}.
\end{equation}
For our given $w_0\in W$, this means there exists $u_0\in U$ and $v_0\in V$ such that $w_0=u_0+v_0$.  We can also choose $\|v_0\|_V$ to be arbitrarily small.  By one of the results of \cite{besovpaper} (recalled in Section $\ref{Besov spaces}$ below as Theorem $\ref{old local theorem}$) there exists a unique global solution $v(t)\in V$ to the LANS equation with initial data $v_0$.  This result also provides a unique local solution $w$ to the LANS equation such that $w(t)\in W$ and $w(0)=w_0$.  

With this global solution $v$ to the LANS equation, the next step is to derive the following modified version of the LANS equation:
\begin{equation*}\aligned \partial_t u-\triangle u+&\div(u\tensor u+u\tensor v+v\tensor u)+\div(\tau^\alpha(u,u)+2\tau^\alpha(u,v)), 
\\ u(0)&=u_0, ~~ \div u=\div u_0=0,
\endaligned
\end{equation*}
where we recall that $\tau^\alpha$ is defined in equation $(\ref{tau})$.  We will refer to this as the mLANS equation, and it is derived by replacing $u$ in $(\ref{LANS})$ with $u+v$.  This process is explicitly detailed in the beginning of Section $\ref{Local Solutions}$. 

Now that the mLANS equation has been defined, we require the following result.
\begin{theorem}\label{mLANS main theorem}For any $u_0\in U$, there exists a unique global solution $u\in C([0,\infty):U)$ to the mLANS equation. 
\end{theorem}

Proving Theorem $\ref{mLANS main theorem}$ will be the primary task of the rest of the article.  For now, assuming Theorem $\ref{mLANS main theorem}$, we proceed with the proof of Theorem $\ref{main theorem}$.  By the construction of the mLANS equation, because $u$ is a global solution to the mLANS equation, we have that $u+v$ is a global solution to the LANS equation, and that $u(0)+v(0)=u_0+v_0=w_0$.  We also have a unique local solution $w$ to the LANS equation with $w(0)=w_0$ and $w(t)\in W$.  By uniqueness, if $u(t)+v(t)\in W$ for all $t$, then $u(t)+v(t)=w(t)$ for all $t$, and the proof of Theorem $\ref{main theorem}$ will be complete.  

So our last remaining task is to show that $u(t)+v(t)\in W$ for all $t$, and this is a special case of a general interpolation result found in \cite{galplan} which will be presented at the end of Section $\ref{a priori sobolev}$.  The key requirement for this result is that $U\hookrightarrow W\hookrightarrow V$.  Using Besov embedding (see equation $(\ref{besov embedding})$), this holds for $U=B^{3/2}_{2,q}(\mathbb{R}^3)$, $V=B^{3/\tilde{p}}_{\tilde{p},q}(\mathbb{R}^3)$, and $W=B^{3/p}_{p,q}(\mathbb{R}^3)$.  Satisfying this embedding relation is the reason we do not use the optimal existence results from \cite{besovpaper} for the interpolation, since $B^{3/4+\varepsilon}_{2,q}(\mathbb{R}^3)$ does not inject into $B^{3/2p+\varepsilon}_{p,q}(\mathbb{R}^3)$.

This completes the proof of Theorem $\ref{main theorem}$, up to proving Theorem $\ref{mLANS main theorem}$, which is the focus of the remainder of the article.  In Section $\ref{Besov spaces}$ we recall the basic construction of Besov spaces, some standard Besov space estimates, and local and global existence theorems from \cite{besovpaper}.  The mLANS equation is derived and local solutions for the mLANS equation are constructed in Section $\ref{Local Solutions}$, and the extension to a global result is the focus of Section $\ref{extension to global existence}$ and Section $\ref{a priori sobolev}$.

\section{Besov Spaces}\label{Besov spaces}

We begin by defining the Besov spaces $B^s_{p,q}(\mathbb{R}^n)$.  Let $\psi_0\in\mathcal{S}$ be an even, radial function with Fourier transform $\hat{\psi_0}$ that has the following properties:
\begin{equation*}\aligned &\hat{\psi_0}(x)\geq 0
\\  &\text{support~}\hat{\psi_0}\subset A_0:=\{\xi\in \mathbb{R}^n:2^{-1}<|\xi|<2\}
\\ &\sum_{j\in\mathbb{Z}} \hat{\psi_0}(2^{-j}\xi)=1, ~\text{for all}~ \xi\neq 0.
\endaligned
\end{equation*}

We then define 
$\hat{\psi_j}(\xi)=\hat{\psi}_0(2^{-j}\xi)$ (from Fourier inversion, this also means $\psi_j(x)=2^{jn}\psi_0(2^jx)$), and remark that $\hat{\psi_j}$ is supported in $A_j:=\{\xi\in\mathbb{R}^n:2^{j-1}<|\xi|<2^{j+1}\}$.  We also define $\Psi$ by 
\begin{equation}\label{low freq part}\hat{\Psi}(\xi)=1-\sum_{k=0}^\infty \hat{\psi}_k(\xi).
\end{equation}

We define the Littlewood Paley operators $\triangle_j$ and $S_j$ by
\begin{equation*}\triangle_j f=\psi_j\ast f, \quad
S_jf=\sum_{k=-\infty}^{j}\triangle_k f,
\end{equation*}
and record some properties of these operators.  Applying the Fourier Transform and
recalling that $\hat{\psi}_j$ is supported on $2^{j-1}\leq
|\xi|\leq2^{j+1}$, it follows that   
\begin{equation}\aligned \label{besovlemma1}\triangle_j\triangle_k f= 0, \quad |j-k|\geq 2
\\ \triangle_j (S_{k-3}f\triangle_{k}g)= 0 \quad |j-k|\geq 4,
\endaligned
\end{equation}
and, if $|i-k|\leq 2$, then 
\begin{equation}\label{besovpieces67}\triangle_j(\triangle_kf\triangle_i g)=0 \quad j>k+4.
\end{equation}

For $s\in\mathbb{R}$ and $1\leq p,q\leq \infty$ we define
the space $\tilde{B}^s_{p,q}(\mathbb{R}^n)$ to be the set of distributions such that 
\begin{equation*}\|u\|_{\tilde{B}^s_{p,q}}=\(\sum_{j=0}^\infty (2^{js}\|\triangle_j
u\|_{L^p})^q\)^{1/q}<\infty,
\end{equation*}
with the usual modification when $q=\infty$.  Finally, we define the Besov spaces $B^s_{p,q}(\mathbb{R}^n)$ by the norm 
\begin{equation*}\|f\|_{B^s_{p,q}}=\|\Psi*f\|_p+\|f\|_{\tilde{B}^s_{p,q}},
\end{equation*}
for $s>0$.  For $s>0$, we define $B^{-s}_{p',q'}$ to be the dual
of the space $B^s_{p,q}$, where $p',q'$ are the Holder-conjugates to
$p,q$.

These Littlewood-Paley operators are also used to define Bony's paraproduct.  We have 
\begin{equation}\label{lp start}fg=\sum_{k} S_{k-3}f\triangle_k g + \sum_{k}S_{k-3}g\triangle_k f+ \sum_{k}\triangle_k f\sum_{l=-2}^2 \triangle_{k+l} g.
\end{equation}

The estimates $(\ref{besovlemma1})$ and $(\ref{besovpieces67})$ imply that 
\begin{equation}\aligned \label{bony256}\triangle_j (fg)\leq &\sum_{k=-3}^3 \triangle_j (S_{j+k-3}f\triangle_{j+k} g)+ \sum_{k=-3}^3 \triangle_j (S_{j+k-3}g\triangle_{j+k} f)
\\ +&\sum_{k>j-4}\triangle_j \(\triangle_k f\sum_{l=-2}^2 \triangle_{k+l}g\).
\endaligned
\end{equation}
This calculation will be very useful in Section $\ref{A Modified Product Estimate}$.

Now we turn our attention to establishing some basic Besov space estimates.  First, we let $1\leq q_1\leq q_2\leq \infty$, $\beta_1\leq \beta_2$, $1\leq p_1\leq p_2\leq\infty$, $\gamma_1=\gamma_2+n(1/p_1-1/p_2)$, and $r>s>0$.  Then we have the following:
\begin{equation}\label{besov embedding}\aligned \|f\|_{B^{\beta_1}_{p,q_2}}&\leq C\|f\|_{B^{\beta_2}_{p,q_1}},
\\ \|f\|_{B^{\gamma_2}_{p_2,q}}&\leq C\|f\|_{B^{\gamma_1}_{p_1,q}},
\\ \|f\|_{H^{s,p}}&\leq \|f\|_{B^r_{p,q}}, 
\\ \|f\|_{H^{s,2}}&=\|f\|_{B^s_{2,2}}\leq \|f\|_{B^{r}_{2,q}}.
\endaligned
\end{equation}
These will be referred to as the Besov embedding results.  Next, we record a Leibnitz-rule type estimate.  This can be found in \cite{chemin}, and for the reader's convenience, the proof can be found in Section $\ref{A Modified Product Estimate}$.
\begin{proposition}\label{new product estimate} Let $f\in B^{s_1}_{p_1,q}(\mathbb{R}^n)$ and let $g\in B^{s_2}_{p_2,q}(\mathbb{R}^n)$.  Then, for any $p$ such that $1/p\leq 1/p_1+1/p_2$ and with $s=s_1+s_2-n(1/p_1+1/p_2-1/p)$,  we have 
\begin{equation*}\|fg\|_{B^s_{p,q}}\leq \|f\|_{B^{s_1}_{p_1,q}}\|g\|_{B^{s_2}_{p_2,q}},
\end{equation*}
provided $s_1<n/p_1$, $s_2<n/p_2$, and $s_1+s_2>0$.
\end{proposition}

Our third result is the Bernstein inequalities (see Appendix $A$ in \cite{taonde}).  We let $A=(-\triangle)$, $\alpha\geq 0$, and $1\leq p\leq q\leq\infty$.  If $\text{supp}~\hat{f}\subset\{\xi\in\mathbb{R}^n:|\xi|\leq
2^jK\}$ and $\text{supp}~\hat{g}\subset\{\xi\in\mathbb{R}^n:2^jK_1\leq |\xi|\leq
2^jK_2\}$ for some $K, K_1, K_2>0$ and some integer $j$, then
\begin{equation}\label{Bernstein}\aligned \tilde{C}2^{j\alpha+jn(1/p-1/q)}\|g\|_p&\leq \|A^{\alpha/2}g\|_q\leq
C2^{j\alpha+jn(1/p-1/q)}\|g\|_p.        
\\ \|A^{\alpha/2} f\|_q&\leq C2^{j\alpha+jn(1/p-1/q)}\|f\|_p
\endaligned 
\end{equation}

Our last Besov space estimate governs the behavior of the heat kernel on Besov spaces.

\begin{proposition}\label{hkb} Let $1\leq p_1\leq p_2<\infty$,
$-\infty<s_1\leq s_2<\infty$, and let $0<q<\infty$.  Then
\begin{equation*}\|e^{t\triangle} f\|_{B^{s_2}_{p_2,q}}\leq
Ct^{-(s_2-s_1+n/p_1-n/p_2)/2}\|f\|_{B^{s_1}_{p_1,q}},
\end{equation*}
provided $0<t<1$.
\end{proposition}

Using the Sobolev space heat kernel estimate, we get, for $0<t<1$, 
\begin{equation*}\aligned \|e^{t\triangle}f\|_{B^{s_2}_{p_2,q}}
&=\|\Psi*e^{t\triangle} f\|_{L^{p_2}}+\(\sum
(2^{js_1}\|2^{j(s_2-s_1)}\triangle_j e^{t\triangle}f\|_{L^{p_2}})^q\)^{1/q}
\\ &\leq t^{(n/p_1-n/p_2)/2}\|\Psi*f\|_{L^{p_1}}+\(\sum
(2^{js_1}\|e^{t\triangle}\triangle_j f\|_{H^{s_2-s_1,p_2}})^q\)^{1/q}
\\ &\leq t^{-(n/p_1-n/p_2)/2}\|\Psi * f\|_{L^{p_1}}+t^{\sigma}\(\sum
(2^{js_1}\|\triangle_j * f\|_{L^{p_1}})^q\)^{1/q}
\\ &\leq t^{\sigma}\|f\|_{B^{s_1}_{p_1,q}}.
\endaligned
\end{equation*}
where $\sigma=-(s_2-s_1+n/p_1-n/p_2)/2$, and we made liberal use of the fact that $e^{t\triangle}$ commutes with convolution operators.
We remark that a straightforward density argument can be used to show that, for any $\varepsilon$, 
\begin{equation}\label{hkb ext}\sup_{0\leq t<T}t^{(s_2-s_1+n/p_1-n/p_2)/2}\|e^{t\triangle} f\|_{B^{s_2}_{p_2,q}}<\varepsilon,
\end{equation}
where $T$ depends only on $\|f\|_{B^{s_1}_{p_1,q}}$.

We conclude this section with the local existence result for the LANS equation from \cite{besovpaper}, but first we define two function spaces.  The first is $BC([0,T):B^{r}_{p,q}(\mathbb{R}^n)$.  We say that $u\in C(X:Y)$ if $u$ is a continuous function from the Banach space $X$ to the Banach space $Y$.  We say that $u\in BC(X:Y)$ if, in addition, $u$ is bounded.  So $u\in BC([0,T):B^{r}_{p,q}(\mathbb{R}^n)$ means that $u$ is continuous from the time interval $[0,T)$ into $B^r_{p,q}(\mathbb{R}^n)$ and  
\begin{equation}\label{defn of bc}\sup_{0\leq t<T}\|u(t)\|_{B^r_{p,q}}:=\|u\|_{0;r,p,q}=M<\infty.
\end{equation}

For $T>0$ and $a\geq 0$, the space $C^T_{a;s,p,q}$ is defined by 
\begin{equation*}C^T_{a;s,p,q}=\{f\in
C((0,T):B^s_{p,q}(\mathbb{R}^n)):\|f\|_{a;s,p,q}<\infty\},
\end{equation*}
where
\begin{equation*}\|f\|_{a;s,p,q}=\sup\{t^a\|f(t)\|_{s,p,q}:t\in
(0,T)\}.
\end{equation*}
We let $\Cdot^T_{a;s,p,q}$ denote the subspace of $C^T_{a;s,p,q}$ consisting
of $f$ such that
\begin{equation*}\lim_{t\rightarrow 0^+}t^a f(t)=0
~\text{(in}~B^s_{p,q}(\mathbb{R}^n)).
\end{equation*}
Note that while the norm $\|\cdot\|_{a;s,p,q}$ lacks an explicit reference to $T$, there is an implicit $T$ dependence. 
Finally, we state a local existence theorem for the LANS equation. This result is a special case of Theorem $4$ in \cite{besovpaper}.

\begin{theorem}\label{old local theorem}Let $v_0\in B^{n/p}_{p,q}(\mathbb{R}^n)$ be divergence free, where $p>n$, and let $r$ satisfy $n/p<r<n/p+1$.  Then there exists a time $T$ and a unique solution $v$ to the LANS equation $(\ref{LANS})$ such that 
\begin{equation*}v\in BC([0,T):B^{n/p}_{p,q}(\mathbb{R}^n)\cap \dot{C}^T_{\frac{r-n/p}{2};r,p,q},
\end{equation*}
with $v(0)=v_0$.  We remark that the time $T$ depends only on $\|v_0\|_{B^{n/p}_{p,q}}$, and for sufficiently small $\|v_0\|_{B^{n/p}_{p,q}}$, $T=\infty$.  Furthermore, for a given $T^*<\infty$ and a given real number $\varepsilon$, if $\|v_0\|_{B^{n/p}_{p,q}}$ is sufficiently small, then $\sup_{0\leq t<T^*}\|v(t)\|_{B^{n/p}_{p,q}}<\varepsilon$.  
\end{theorem}

The result can be extended in the following fashion.  
\begin{corollary}\label{old local extension}Let $v_0\in B^{n/p}_{p,q}(\mathbb{R}^n)$ be divergence free, and let $v$ be the solution given in the above theorem.  Then the requirement that $r<n/p+1$ can be removed.
\end{corollary}
The proof of a similar extension for solutions to the mLANS equation can be found in Lemma $\ref{higher reg}$ in Section $6$.  A more complete discussion of this type of result can also be found there.  

\section{Derivation of and Local Solutions to the mLANS equation}\label{Local Solutions}

We let $v(t)$ denote the solution to the LANS equation with $v(0)=v_0$ given by Theorem $\ref{old local theorem}$.  We seek a $u$ such that, defining $w(t)$ by $w(t)=u(t)+v(t)$, $w$ will solve the LANS equation.  This means 
\begin{equation*}\aligned &\partial_t (u+v)-\triangle (u+v) +\div((u+v)\tensor (u+v) + \tau^\alpha(u+v,u+v)).
\endaligned
\end{equation*}

Using the fact that $v$ satisfies the LANS equation, and requiring that $u(0)=u_0$ and $\div u=\div u_0=0$, this (essentially) simplifies to 
\begin{equation}\aligned \label{new pde}\partial_t u-\triangle u+&\div(u\tensor u+u\tensor v+v\tensor u)+\div(1-\alpha^2\triangle)^{-1}(\nabla u\nabla u+\nabla u\nabla v),
\\ &u(0)=u_0, ~~ \div u=\div u_0=0.
\endaligned
\end{equation}
This is not exact because of the second non-linear term.  There are actually several more terms involving products of $\nabla u$, $(\nabla u)^T$, $\nabla v$, and $(\nabla v)^T$ (but no terms involving only products of $\nabla v$ and $(\nabla v)^T$).  In most of the following calculations, the additional terms have no effect on our argument, and so will often be omitted.  We call equation $(\ref{new pde})$ the mLANS equation.  

Throughout the remainder of the article, we set the $v$ in the mLANS equation to be the small initial data solution to the LANS equation given by Theorem $\ref{old local theorem}$ with $p>>n$, which means $v$ is divergence-free and   
\begin{equation}\label{google}v\in \tilde{E}=BC([0,T):B^{n/p}_{p,q}(\mathbb{R}^n))\cap \dot{C}^T_{b;r,p,q},
\end{equation}
for any $r>n/p$ and any $T$, where $b=(r-n/p)/2$.  

\begin{theorem}\label{new local version 34}Let $u_0\in B^{n/2}_{2,q}(\mathbb{R}^n)$ be divergence free.  Then there exists a local solution $u$ to the mLANS equation $(\ref{new pde})$ such that 
\begin{equation*}u\in BC([0,T):B^{n/2}_{2,q}(\mathbb{R}^n))\cap \dot{C}^T_{a;s,2,q},
\end{equation*}
where $a=(s-n/2)/2$, $0<s-n/2<1$, and $T$ depends only on $\|u_0\|_{n/2,2,q}$. 
\end{theorem}

In the next section, we will extend this local solution to a global solution.  The following Corollary is instrumental in this task.
\begin{corollary}\label{cor34}The requirement in Theorem $\ref{new local version 34}$ that $s-n/2<1$ can be removed.
\end{corollary}
The proof of the corollary follows from Lemma $\ref{higher reg}$ in Section $\ref{Extending the local result: Higher regularity}$.  This result (and its proof) is similar to Proposition $8$ in \cite{sobpaper}, which has its origins in an induction argument from \cite{katoinduction}.

The proof of Theorem $\ref{new local version 34}$ will follow from the standard contraction mapping method and heavy use of the results from Section $\ref{Besov spaces}$.  We begin by defining the nonlinear operator $\Phi$ by 
\begin{equation*}\Phi(u)=e^{t\triangle}u_0+\tilde{\Phi}(u)+\Psi(u,v),
\end{equation*}
where
\begin{equation*}\aligned \tilde\Phi(u)=\int_0^t e^{(t-s)\triangle}(V(u))ds
\\ \Psi(u,v)=\int_0^t e^{(t-s)\triangle}W(u,v)ds,
\endaligned
\end{equation*}
with $V$ and $W$ (essentially) given by  
\begin{equation*}\aligned V(u)=\div(u\tensor u)+\div(1-\triangle)^{-1}(\nabla u\nabla u),
\\ W(u,v)=\div(u\tensor v)+\div(1-\triangle)^{-1}(\nabla u\nabla v),
\endaligned
\end{equation*}
where, as above, the full definitions of $V$ and $W$ involve additional terms whose behavior is controlled by the terms shown.

We seek a fixed point of $\Phi$ in the space  
\begin{equation*}\aligned E&=\{f\in BC((0,T):B^{n/2}_{2,q}(\mathbb{R}^n))\cap \dot{C}^T_{\frac{s-n/2}{2};s,2,q}: 
\\ &\sup_t \|f-e^{t\triangle}u_0\|_{n/2,2,q}+\|f\|_{(s-n/2)/2;s,2,q}<M\}.
\endaligned \end{equation*}

We first show that $\Phi:E\rightarrow E$, and we begin by showing that $\Psi:E\rightarrow E$, which requires estimating   
\begin{equation}\aligned \label{dirk1}I=&I_1+I_2=\sup_{0\leq t<T}\|\int_0^t e^{(t-\tau)\triangle}\div(u\tensor v)d\tau\|_{B^{n/2}_{2,q}} \\ +&\sup_{0<t<T}t^a\|\int_0^t e^{(t-s)\triangle}\div(u\tensor v)d\tau\|_{B^{s}_{2,q}} 
\endaligned
\end{equation}
and 
\begin{equation}\aligned \label{dirk2} J=&J_1+J_2=\sup_{0\leq t<T}\|\int_0^t e^{(t-\tau)\triangle}\div(1-\triangle)^{-1}(\nabla u\nabla v)d\tau\|_{B^{n/2}_{2,q}}
\\ +&\sup_{0<t<T}t^a\|\int_0^t e^{(t-\tau)\triangle}\div(1-\triangle)^{-1}(\nabla u\nabla v)d\tau\|_{B^{s}_{2,q}}
\endaligned
\end{equation}

\subsection{Estimating $I$}
To bound $I_1$, we start by setting $\alpha=\alpha_1+\alpha_2$, where $\alpha_1=n/2-\varepsilon$ and $\alpha_2=n/p-\varepsilon$.  Then we have   
\begin{equation*}\|\int_0^t e^{(t-\tau)\triangle}\div(u\tensor v)d\tau\|_{B^{n/2}_{2,q}}
\leq \int_0^t |t-\tau|^{-({n/2}-(\alpha-1)+n/\bar{p}-n/2)/2}\|u\tensor v\|_{B^{\alpha}_{\bar{p},q}}d\tau,
\end{equation*}
where $n/\bar{p}=n/2+n/p$ and we used that $\alpha-1=n/2+n/p-2\varepsilon-1\leq {n/2}$ for $p>n$.  Using Proposition $\ref{new product estimate}$, we have 
\begin{equation*}\|u\tensor v\|_{B^{\alpha}_{\bar{p},q}}\leq \|u\|_{B^{\alpha_1}_{2,q}}\|v\|_{B^{\alpha_2}_{p,q}}\leq \|u\|_{B^{n/2}_{2,q}}\|v\|_{B^{n/p}_{p,q}}.
\end{equation*}

Returning to the integral, we have 
\begin{equation*}\aligned &\int_0^t |t-\tau|^{-({n/2}-(\alpha-1)+n/\bar{p}-n/2)/2}\|u(\tau)\tensor v(\tau)\|_{B^{\alpha}_{\bar{p},q}}d\tau
\\ \leq C&\int_0^t |t-\tau|^{-(1+2\varepsilon)/2}\|u(\tau)\|_{B^{n/2}_{2,q}}\|v(\tau)\|_{B^{n/p}_{p,q}}d\tau
\\ \leq C&\|u\|_{0;n/2,2,q}\|v\|_{0,n/p,p,q}\int_0^t |t-\tau|^{-(1+2\varepsilon)/2}d\tau
\\ \leq C&\|u\|_{0;n/2,2,q}\|v\|_{0,n/p,p,q}t^{-(1+2\varepsilon)/2+1},
\endaligned
\end{equation*}
provided $1+2\varepsilon<2$, which is easily satisfied for small $\varepsilon$.  From $(\ref{google})$, we know that $\|v\|_{0;n/p,p,q}$ is finite, so 
\begin{equation}\label{vet2}\aligned I_1&\leq \sup_{0\leq t<T}C\|u\|_{0;n/2,2,q}\|v\|_{0,n/p,p,q}t^{-(1+2\varepsilon)/2+1}
\\ &\leq CMT^{1/2-\varepsilon}.
\endaligned
\end{equation}

For $I_2$, recalling that $a=({s}-{n/2})/2$, a similar argument gives 
\begin{equation}\label{vet3}\aligned I_2&\leq \sup_{0<t<T}t^a\|\int_0^t e^{(t-\tau)\triangle}\div(u(\tau)\tensor v(\tau))d\tau\|_{B^{s}_{2,q}} 
\\ &\leq \sup_{0<t<T}t^{a}\int_0^t |t-\tau|^{-({s}-(\alpha-1)+n/\bar{p}-n/2)/2}\|u(\tau)\tensor v(\tau)\|_{B^\alpha_{\bar{p},q}} d\tau
\\ &\leq \sup_{0<t<T}Ct^{a}\|u\|_{0,n/2,2,q}\|v\|_{0,n/p,p,q}\int_0^t |t-\tau|^{-({s}-\alpha+1+n/p)/2}d\tau 
\\ &\leq \sup_{0<t<T}C\|u\|_{0,n/2,2,q}\|v\|_{0,n/p,p,q}t^{-({s}-\alpha+1+n/p)/2+1+a},
\\ &\leq CMT^{(1-n/p-n/2+\alpha)/2}\leq CMT^{1/2-\varepsilon},
\endaligned
\end{equation}
provided ${s}-\alpha+n/p<1$.  So we have that 
\begin{equation}\label{fifai}I=I_1+I_2<CMT^{1/2-\varepsilon},
\end{equation}
provided  
\begin{equation*}\aligned 1&>{s}-\alpha+n/p=s-n/2+2\varepsilon
\\ 1&\geq n/2-\alpha+n/p=2\varepsilon.
\endaligned
\end{equation*}
The first requirement is equivalent to $s-n/2<1$ and the second is vacuously satisfied.

\subsection{Estimating $J$}
For $J_1$, we have  
\begin{equation*}\aligned &\|\int_0^t e^{(t-\tau)\triangle}\div(1-\triangle)^{-1}(\nabla u(\tau) \nabla v(\tau))d\tau\|_{B^{n/2}_{2,q}}
\\ \leq &\int_0^t |t-\tau|^{-(n/\bar{p}-n/2)/2}\|\div(1-\triangle)^{-1}(\nabla u(\tau) \nabla v(\tau))\|_{B^{n/2}_{\bar{p},q}}d\tau
\\ \leq &\int_0^t |t-\tau|^{-(n/\bar{p}-n/2)/2}\|\nabla u(\tau) \nabla v(\tau)\|_{B^{n/2-1}_{\bar{p},q}}d\tau
\endaligned
\end{equation*}
where $n/\bar{p}=n/2+n/p$.  Setting $n/2-1=\beta_1+\beta_2$, where $\beta_1<n/2$ and $\beta_2<n/p$, and again using Proposition $\ref{new product estimate}$, we have 
\begin{equation*}\|\nabla u\nabla v\|_{B^{n/2-1}_{\bar{p},q}}\leq \|u\|_{B^{\beta_1+1}_{2,q}}\|v\|_{B^{\beta_2+1}_{p,q}}\leq \|u\|_{B^{n/2}_{2,q}}\|v\|_{B^{r}_{p,q}},
\end{equation*}
where $r\geq \beta_2+1$.    

Recalling that $b=({r}-n/p)/2$, we get that $J_1$ is bounded by 
\begin{equation*}\aligned J_1&\leq \sup_{0\leq t<T}\leq \int_0^t |t-\tau|^{-(n/\bar{p}-n/2)/2} \|\div(1-\triangle)^{-1}(\nabla u(\tau)\nabla v(\tau))\|_{B^{n/2}_{\bar{p},q}}d\tau
\\ &\leq \sup_{0\leq t<T}\int_0^t |t-\tau|^{-n/2p} \|u(\tau)\|_{B^{n/2}_{2,q}}\|v(\tau)\|_{B^{r}_{p,q}}d\tau
\\ &\leq \sup_{0\leq t<T}C\|u\|_{0;{n/2},2,q}\|v\|_{b;{r},p,q}\int_0^t |t-\tau|^{-n/2p}\tau^{-b}d\tau
\\ &\leq CMT^{-n/2p-b+1}=CMT^{1-r/2},
\endaligned
\end{equation*}
provided $b<1$ (recall that, by equation $(\ref{google})$, $\|v\|_{b;r,p,q}$ is bounded).

For $J_2$, we have 
\begin{equation*}\aligned J_2&\leq \sup_{0<t<T}t^a\|\int_0^t e^{(t-\tau)\triangle}\div(1-\triangle)^{-1}(\nabla u(\tau) \nabla v(\tau))d\tau\|_{B^{s}_{2,q}}
\\ &\leq \sup_{0<t<T}t^a\int_0^t |t-\tau|^{-({s}-n/2+n/\bar{p}-n/2)/2}\|\div(1-\triangle)^{-1}(\nabla u(\tau)\nabla v(\tau))\|_{B^{n/2}_{\bar{p},q}}d\tau
\\ &\leq C\sup_{0<t<T}t^a\|u\|_{0;n/2,2,q}\|v\|_{b,{r},p,q}\int_0^t |t-\tau|^{-({s}-n/2+n/p)/2}\tau^{-b}d\tau
\\ &\leq CMT^{1-r/2},
\endaligned
\end{equation*}
provided ${s}-n/2+n/p<2$ and $b<1$.  

Combining the restrictions, we get that 
\begin{equation*}J=J_1+J_2\leq CMT^{1-r/2},
\end{equation*}
provided  
\begin{equation*}\label{restrictions334}\aligned {n/2}&>\beta_1+1
\\ \beta_2&<n/p
\\ {r}&\geq \beta_2+1
\\ 2&>{s}-n/2+n/p
\\ 2&\geq {r}.
\endaligned
\end{equation*}

Setting $\beta_1=n/2-1-n/2p$, $\beta_2=n/2p$, and $r=\beta_2+1$, we get  
\begin{equation}\label{fifaj}J=J_1+J_2\leq CMT^{1/2-n/4p}
\end{equation}
provided $s-n/2<1$.    

\subsection{Finishing Theorem $\ref{new local version 34}$}
From equations $(\ref{fifai})$ and $(\ref{fifaj})$, we have that 
\begin{equation*}\|\Psi\|_E\leq I+J <CM(T^{1/2-n/4p}+T^{1/2-\varepsilon})<CMT^{1/2-n/4p},
\end{equation*}
for $p>>n$.  For $\tilde{\Phi}$, similar calculations yield 
\begin{equation*}\|\tilde{\Phi}\|_E\leq CM^2.
\end{equation*}
Thus 
\begin{equation*}\|\tilde{\Phi}\|_E+\|\Psi\|_E\leq M/2,
\end{equation*}
provided $M$ and $T$ are sufficiently small.  We remark that the size of $T$ required here depends only on the parameters and on constants, not on $u$ or $M$.  For the linear term $e^{t\triangle}u_0$, Proposition $\ref{hkb}$ and equation $(\ref{hkb ext})$ give that 
\begin{equation*}\|e^{t\triangle}u_0\|_E=\|e^{t\triangle}u_0\|_{a;s,2,q}<M/2,
\end{equation*}
provided $T$ is sufficiently small.  We note that the desired $T$ depends only on $M$ and $\|u_0\|_{B^{n/2}_{2,q}}$. So we have that $\Phi:E\rightarrow E$ provided $M$ and $T$ are sufficiently small, and $T$ can be taken taken as a function of $\|u_0\|_{B^{n/2}_{2,q}}$.  The proof that $\Phi$ is a contraction follows from the standard contraction mapping argument and will be omitted.

\section{Extension to Global existence}\label{extension to global existence}
In this section, we prove the following Theorem.
\begin{theorem}\label{fozzy1}The solution $u$ to the mLANS equation given by Theorem $\ref{new local version 34}$ with initial data $u_0\in B^{3/2}_{2,q}(\mathbb{R}^3)$ can be extended to a global solution.
\end{theorem}

The proof follows from a bootstrapping argument and $\emph{a~priori}$ estimates proven in the next section.  We begin here by setting up the bootstrap, and start by assuming the unique local solution $u$ with $u(0)=u_0\in B^{3/2}_{2,q}(\mathbb{R}^3)$ given by Theorem $\ref{new local version 34}$ satisfies $u\in BC([0,T_0):B^{3/2}_{2,q}(\mathbb{R}^3))$ for some $T_0<\infty$.  By definition (equation $(\ref{defn of bc})$), this means 
\begin{equation*}\sup_{0\leq t<T_0}\|u(t)\|_{B^{3/2}_{2,q}}=M<\infty.
\end{equation*}
For any $t\in [0,T_0)$, define $v_t(0)=u(t)$.  Then by Theorem $\ref{new local version 34}$ there is a unique solution to the mLANS equation $v_t\in BC([0,T(\|v_t(0)\|_{B^{3/2}_{2,q}})):B^{3/2}_{2,q}(\mathbb{R}^3))$ with initial data $v_t(0)=u(t)$, where $T(\|v_t(0)\|_{B^{3/2}_{2,q}})$ indicates that the time interval of the solution depends only on  $\|v_t(0)\|_{B^{3/2}_{2,q}}$.   The key fact here is that, since $\|u(t)\|_{B^{3/2}_{2,q}}=\|v_t(0)\|_{B^{3/2}_{2,q}} \leq M$ for any $t$, there exists a $\tilde{T}$ such that $T(\|v_t(0)\|_{B^{3/2}_{2,q}})\geq \tilde{T}$, and thus, for any $t$, $v_t$ exists on a time interval of at least length $\tilde{T}$.  

By uniqueness, we also have that $v_t(s)=u(t+s)$, for $s\in [0,\tilde{T})$, which means $u$ exists on the interval $[t,t+\tilde{T})$ for any $t\in [0,T)$.  By choosing $t^*=T_0-\tilde{T}/2$, the original solution $u$ is extended to $u\in BC([0,T_1):B^{n/2}_{2,q}(\mathbb{R}^n))$, where $T_1=T_0+\tilde{T}/2$.  This completes the bootstrap.  

By this bootstrapping argument, given that 
\begin{equation}\label{fozzy2}\sup_{0\leq t<T}\|u(t)\|_{B^{3/2}_{2,q}}=M<\infty,
\end{equation}
the local solution $u$ can be extended to a time interval $[0,T_1)$, where $T<T_1$.   

To prove Theorem $\ref{fozzy1}$, we will assume for contradiction that our solution $u$ is not a global solution.  This means there exists a $T^*<\infty$ such that $u\in BC([0,T):B^{3/2}_{2,q}(\mathbb{R}^3))$ for any $T<T^*$, but $u\notin BC([0,T^*):B^{3/2}_{2,q}(\mathbb{R}^3))$.
  
To contradict this assumption, we use the following $\emph{a~priori}$ results.  We first recall, by Corollary $\ref{old local extension}$, that $u(t)\in B^{r}_{2,q}(\mathbb{R}^3)$ for any real $r$.  Then by Besov embedding (equation $(\ref{besov embedding}$)), we have 
\begin{equation}\aligned \label{mkg1}\|u(t)\|_{B^{3/2}_{2,q}}&\leq \|u(t)\|_{B^{2}_{2,2}}=\|u(t)\|_{H^{2,2}},
\\ \|u(t)\|_{H^{3,2}}&\leq \|u(t)\|_{B^{3+\varepsilon}_{2,q}},
\endaligned
\end{equation}
for any $q\in [1,\infty)$.  This means $u$ satisfies the hypothesis of Theorem $\ref{mlansglobalthm3}$ (proven in Section $\ref{a priori sobolev}$ below), so by Theorem $\ref{mlansglobalthm3}$, for any $a\in (0,T^*)$, 
\begin{equation}\label{mkg2}\sup_{a\leq t<T^*}\|u(t)\|_{B^{3/2}_{2,q}}=K<\infty.
\end{equation} 

By assumption, since $a<T$, we have that $u\in BC([0,a):B^{3/2}_{2,q}(\mathbb{R}^3))$, so we have finally proven that $u\in BC([0,T^*):B^{3/2}_{2,q}(\mathbb{R}^3))$, which provides the desired contradiction, and finishes the proof of Theorem $\ref{fozzy1}$, up to the proof of Theorem $\ref{mlansglobalthm3}$, which is the main result of the next section.

\section{Sobolev space $\emph{a~priori}$ estimates}\label{a priori sobolev}
As mentioned in the introduction, it is easier to control the long time behavior of solutions to the LANS equation than solutions to the Navier-Stokes equation.  More specifically, cancellation in the non-linear terms leads to uniform-in-time bounds on the Sobolev space norms of the solution, which, combined with standard bootstrapping arguments, can extend local solutions to global solutions.  The first goal of this section is to prove Theorem $\ref{mlansglobalthm3}$, an analogous $\emph{a~priori}$ bound for the mLANS equation, which completes the proof of Theorem $\ref{fozzy1}$.  At the end of this section, we address the abstract interpolation result referenced at the end of the introduction.  These results complete the proof of Theorem $\ref{mlansglobalthm2}$.  

Before stating the $\emph{a~priori}$ results, we recall some notation from the previous section and some properties of our small-data global solution $v$ that will be used throughout this section.  For notational convenience, we set $(-\triangle)=A$.  We let $T^*$ be as in the previous section, and we let $a\in (0,T^*)$.  From Theorem $\ref{old local theorem}$, Corollary $\ref{old local extension}$, and the Besov space embedding results (equation $(\ref{besov embedding}))$, we have that
\begin{equation}\label{uniform v bound}\sup_{a\leq t<T^*}\|v(t)\|_{H^{r,p}}\leq N<\infty,
\end{equation}
and that, for any $\varepsilon>0$,  
\begin{equation}\label{uniform initial reg bound}\sup_{0\leq t<T^*}\|v(t)\|_{B^{n/p}_{p,q}}<\varepsilon,
\end{equation}
provided $\|v_0\|_{B^{n/p}_{p,q}}$ is small enough.

The first $\emph{a~priori}$ result provides a bound for the $H^{1,2}(\mathbb{R}^n)$ norm of a solution $u$ to the mLANS equation.

\begin{lemma}\label{mlanssobglobal1}Let $u$ be a solution to the mLANS equation, with $v$ as described above.  Then 
\begin{equation*}\aligned \sup_{a\leq t<T^*}\|u(t)\|_{L^2}^2+\alpha^2\|u(t)\|_{\dot{H}^{1,2}}^2\leq C(N)(\|u(a)\|_{L^2}^2+\alpha^2\|u(a)\|_{\dot{H}^{1,2}}^2),
\endaligned
\end{equation*}
where $\|\cdot\|_{\dot{H}^{r,p}}$ denotes the homogeneous Sobolev space norm.
\end{lemma}

Note that, if $\alpha=0$, this only provides an $L^2$ bound, which is not sufficient to extend teh local solutions to global solutions.  The second lemma provides a bound for the $\dot{H}^{2,2}(\mathbb{R}^3)$ norm.  
\begin{lemma}\label{mlansglobalthm2}Let $u$ be a solution to the mLANS equation, with $v$ as specified in the beginning of the section.  We assume $u(t)\in \dot{H}^{3,2}(\mathbb{R}^3)$ for any $t$ in $[a,T^*)$.  Then  
\begin{equation}\label{noire1}\sup_{t\in [a,T^*)}\|u(t)\|_{\dot{H}^{2,2}}= K<\infty
\end{equation}
for some real number $K$.  
\end{lemma}

The combination of these two Lemma's and the Besov embeddings in equation $(\ref{besov embedding})$ proves the following Theorem.
\begin{theorem}\label{mlansglobalthm3}Let $u$ be a solution to the mLANS equation, with $v$ as specified in the beginning of the section.  We assume $u(t)\in H^{3,2}(\mathbb{R}^3)$ for any $t$ in $[a,T^*)$.  Then  
\begin{equation}\label{noire4}\sup_{t\in [a,T^*)}\|u(t)\|_{B^{3/2}_{2,q}}= K<\infty
\end{equation}
for some real number $K$.
\end{theorem}

\subsection{Proof of Lemma $\ref{mlanssobglobal1}$} 
We begin the proof of the Lemma by stating the following equivalent form of the mLANS equation (see Section $3$ of \cite{MS}):
\begin{equation}\label{LANSglobalpi}\aligned &\partial_t (1+A\alpha^2)u(t) +(1+A\alpha^2)A u(t)
=-\nabla p -\alpha^2(\nabla {u(t)})^T\cdot A{u(t)}
\\  &-\nabla_{u(t)}[(1+A\alpha^2){u(t)}]-(1+A\alpha^2)(\div({u(t)}\tensor {v(t)}))- \div(\nabla {u(t)}\nabla {v(t)})
\endaligned
\end{equation}
Taking the $L^2$ product of the equation with ${u(t)}$, we get 
\begin{equation*}\aligned \label{monticeto}&\partial_t (\|u(t)\|_{L^2}^2+\alpha^2\|A^{1/2}u(t)\|_{L^2}^2) +\|A^{1/2}u(t)\|_{L^2}^2+\alpha^2\|Au(t)\|_{L^2}^2
\\ \leq &I_1+I_2+I_3+I_4+I_5+I_6,
\endaligned
\end{equation*}
where
\begin{equation*}\aligned I_1&=(\nabla_{u(t)} {u(t)},{u(t)}),
\\ I_2&=\alpha^2\((\nabla_{u(t)} \triangle {u(t)},{u(t)})+((\nabla {u(t)})^T\cdot A{u(t)}, {u(t)})\),
\\ I_3&=(\nabla p,{u(t)}),
\\ I_4&=(\div({u(t)}\tensor {v(t)}),{u(t)}),
\\ I_5&=\alpha^2(A\div({u(t)}\tensor {v(t)}),{u(t)}),
\\ I_6&=(\div(\nabla {u(t)}\nabla {v(t)}),{u(t)}).
\endaligned
\end{equation*}

An application of integration by parts and recalling that $\div {u(t)}=0$ gives that $I_1=I_3=0$.  For $I_2$, writing it in coordinates (and temporarily suppressing the time dependence), we see that 
\begin{equation*}\aligned I_2=&\sum_{i,j=1}^3\alpha^2\int u_i(\partial_{x_i} \triangle u_j)u_j+(\triangle u_i)(\partial_{x_j} u_i)u_j
 \\ =&\sum_{i,j=1}^3\alpha^2\int -(u_i(\triangle u_j)(\partial_{x_i} u_j))+(\triangle u_i)(\partial_{x_j} u_i)u_j
= 0,
\endaligned
\end{equation*}
where we again used integration by parts and exploited the divergence free condition.  We remark here that it is these cancellations which make it easier to control the long time behavior of the LANS equations.  For $I_4$, using Holder's inequality and the Sobolev embedding theorem (and recalling that $\|\cdot\|_{\dot{H}^{s,p}}$ denotes the homogeneous Sobolev space norm), we have 
\begin{equation}\label{rubio1}\aligned I_4&\leq \|\nabla {u(t)}\|_{L^2}\|{u(t)}\tensor {v(t)}\|_{L^2}\leq \|\nabla {u(t)}\|_{L^2}\|{u(t)}\|_{L^{\tilde{p}}}\|{v(t)}\|_{L^p}
\\ &\leq \|\nabla {u(t)}\|_{L^2}\|{u(t)}\|_{\dot{H}^{s,2}}\|{v(t)}\|_{L^p},
\endaligned
\end{equation}
where Holder's inequality requires $1/2=1/\tilde{p}+1/p$ and the Sobolev embedding theorem requires $1/\tilde{p}=1/2-s/3$, with $2s<3$.  Solving the system for $s$, we get that $s=3/p$, and for $p>3$, we have that $s<1$, so we finally bound $I_4$ by 
\begin{equation*}I_4\leq C\alpha^{-2}(\|u(t)\|^2_{L^2}+\alpha^2\|\nabla {u(t)}\|_{L^2}^2)\|{v(t)}\|_{L^p}.
\end{equation*}

To bound $I_5$, we use integration by parts, the Leibnitz rule, and then Holder's inequality and Sobolev embeddings as in the estimate of $I_4$ to get 
\begin{equation*}\aligned I_5&\leq \alpha^2((A{u(t)}){v(t)},\nabla {u(t)})+(\nabla {u(t)} \nabla {v(t)},\nabla {u(t)})+({u(t)}(\triangle {v(t)}),\nabla {u(t)})
\\ &\leq \alpha^2(\|{u(t)}\|_{\dot{H}^{2,2}}\|{v(t)}\nabla {u(t)}\|_{L^2}+\|\nabla {u(t)}\|_{L^2}\|\nabla {u(t)} \nabla {v(t)}\|_{L^2}+ \|\nabla {u(t)}\|_{L^2}\|{u(t)}Av(t)\|_{L^2})
\\ &\leq C\alpha^2(\|{u(t)}\|_{\dot{H}^{2,2}}\|{v(t)}\|_{L^p}+\|\nabla {u(t)}\|_{L^2}^2\|\nabla {v(t)}\|_{L^\infty}+\|\nabla {u(t)}\|_{L^2}^2\|A {v(t)}\|_{L^p})
\\ &\leq C\alpha^2(\|{u(t)}\|_{\dot{H}^{2,2}}\|{v(t)}\|_{L^p})+(\|{u(t)}\|_{L^2}^2+\alpha^2\|\nabla {u(t)}\|_{L^2}^2)\|{v(t)}\|_{\dot{H^{2,p}}}.
\endaligned
\end{equation*}  

For $I_6$, the same type of argument gives 
\begin{equation*}\aligned I_6&\leq ((\nabla {u(t)}\nabla {v(t)}),\nabla {u(t)})\leq \|\nabla u(t)\|_{L^2}^2\|v(t)\|_{\dot{H}^{1+n/p+\varepsilon,p}}
\\ &\leq C\alpha^{-2}(\|{u(t)}\|_{L^2}^2+\alpha^2\|\nabla {u(t)}\|_{L^2}^2)\|{v(t)}\|_{\dot{H^{2,p}}}.
\endaligned
\end{equation*}

Plugging the estimates for $I_1$ through $I_6$ back into equation $(\ref{monticeto})$, we get 
\begin{equation}\label{hyu}\aligned \partial_t &(\|{u(t)}\|_{L^2}^2+\alpha^2\|{u(t)}\|_{L^2}^2)\leq J_1+J_2+J_3,
\endaligned
\end{equation}
where 
\begin{equation*}\aligned J_1&=C\alpha^{-2}(\|u(t)\|^2_{L^2}+\alpha^2\|\nabla  {u(t)}\|_{L^2}^2)(\|{v(t)}\|_{L^p}+\|{v(t)}\|_{\dot{H^{2,p}}}),
\\ J_2&=\alpha^2C\|{u(t)}\|_{\dot{H}^{2,2}}(C\|{v(t)}\|_{L^p}-1),
\\ J_3&=-\|\nabla {u(t)}\|_{L^2}^2.
\endaligned
\end{equation*}

From equation $(\ref{uniform initial reg bound})$, for sufficiently small $\|v_0\|_{B^{3/p}_{p,q}}$, we have that  $C\|{v(t)}\|_{L^p}-1<0$ for all $t$.  This makes $J_2$ and $J_3$ negative, so $(\ref{hyu})$ becomes 
\begin{equation*}\aligned \partial_t &(\|{u(t)}\|_{L^2}^2+\alpha^2\|{u(t)}\|_{L^2}^2)
\\ \leq &C\alpha^{-2}(\|u(t)\|^2_{L^2}+\alpha^2\|\nabla  {u(t)}\|_{L^2}^2)(\|{v(t)}\|_{L^p}+
\|{v(t)}\|_{\dot{H^{2,p}}}).
\endaligned
\end{equation*}

Applying Gronwall's inequality gives  
\begin{equation*}\aligned &\|u(t)\|_{L^2}^2+\alpha^2\|u(t)\|_{L^2}^2
\\ \leq &(\|u(a)\|_{L^2}^2+\alpha^2\|u(a)\|_{L^2}^2) \exp\{C\alpha^{-2}\int_a^{T^*} \|v(s)\|_{H^{2,p}}ds\},
\endaligned
\end{equation*}
and an application of equation $(\ref{uniform v bound})$ completes the Lemma.  We observe that this result could be extended to higher dimensions by taking more care with the Sobolev embeddings.

\subsection{Proof of Lemma $\ref{mlansglobalthm2}$}
We first observe that, to prove Lemma $\ref{mlansglobalthm2}$, it is sufficient to take the supremum over all $t$ such that $\|u(t)\|_{\dot{H}^{2,2}}$ and $\|u(t)\|_{\dot{H}^{3,2}}$ are greater than one.  For the proof, we start with the standard form of the mLANS equation, apply $A$ to both sides and take the $L^2$ product with $Au$ to get
\begin{equation}\label{aldo}\aligned &(\partial_t A{u(t)},A{u(t)})+(A^2{u(t)},A{u(t)})
=I+J,
\endaligned
\end{equation}
where 
\begin{equation*}\aligned I&=-(AP^\alpha(\nabla_{u(t)} {u(t)}+\div(1+A\alpha^2)^{-1}(\nabla {u(t)} \nabla {u(t)})),A{u(t)}),
\\ J&= -(AP^\alpha(\div({u(t)}\tensor v)+\div(1+A\alpha^2)^{-1}(\nabla {u(t)}\nabla {v(t)})),A{u(t)}).
\endaligned
\end{equation*}
For the left hand side, we have 
\begin{equation}\label{aldo1}(\partial_t A{u(t)},A{u(t)})=\frac{1}{2}\partial_t \|{u(t)}\|^2_{\dot{H}^{2,2}}
\end{equation}
and 
\begin{equation}\label{aldo2}(A^2 {u(t)},A{u(t)})=(A^{3/2}{u(t)},A^{3/2}{u(t)})=\|{u(t)}\|^2_{\dot{H}^{3,2}}.
\end{equation}
Estimating $I$ and $J$ is significantly harder, and is the subject of the next two subsections.

\subsubsection{Estimating $I$}\label{Estimating I}
We start by re-writing $I$ as $I=K_1+K_2$, where 
\begin{equation*}\aligned K_1=-(AP^\alpha(\nabla_{u(t)} {u(t)}),A{u(t)}),
\\ K_2=-(AP^\alpha\div\tau^\alpha {u(t)},A{u(t)}).
\endaligned
\end{equation*}

We will make heavy use of the following Ladyzhenskaya inequality ($(5.3)$ in \cite{MS}) which holds in $\mathbb{R}^3$:
\begin{equation}\label{AgLa}\aligned
\|f\|_{\dot{H}^{r_1,2}}&\leq C\|f\|^{1-r_1/r_2}_{L^2}\|f\|^{r_1/r_2}_{\dot{H}^{r_2,2}},
\endaligned
\end{equation}

Starting with $K_1$, making liberal use of integration by parts, the product rule, and Holder's inequality, we have
\begin{equation}\label{glob1d}\aligned |K_1|&\leq (A^{1/2}(\nabla_{u(t)} {u(t)}),A^{3/2}{u(t)})
\\ &\leq C\|A^{3/2}{u(t)}\|_{L^2}(\|(A^{1/2}\nabla {u(t)}){u(t)}\|_{L^2}+\|(A^{1/2}{u(t)})\nabla {u(t)}\|_{L^2})
\\ &\leq C\|{u(t)}\|_{\dot{H}^{3,2}} (\|{u(t)}\|_{L^\infty} \|{u(t)}\|_{\dot{H}^{2,2}}+\|A^{1/2}{u(t)}\|_{L^\infty}\|{u(t)}\|_{\dot{H}^{1,2}}).
\endaligned
\end{equation}
By Sobolev embedding, we have
\begin{equation}\label{agla2}\aligned \|u\|_{L^\infty}&\leq \|{u(t)}\|_{L^2}+\|{u(t)}\|_{\dot{H}^{k_1}},
\\ \|A^{1/2}{u(t)}\|_{L^\infty}&\leq \|\nabla {u(t)}\|_{L^2}+\|{u(t)}\|_{\dot{H}^{k_2}},
\endaligned
\end{equation}
provided $k_1=3/2+\varepsilon$ and $k_2=5/2+\varepsilon$ for positive $\varepsilon$.  Recalling that Lemma $\ref{mlanssobglobal1}$ provides a uniform bound of $M$ on $\|u(t)\|_{H^{1,2}}$, we can now bound $K_1$ by 
\begin{equation}\label{glob1c}\aligned |K_1|\leq &C\|{u(t)}\|_{\dot{H}^{3,2}}\|{u(t)}\|_{\dot{H}^{2,2}}(M+\|{u(t)}\|_{\dot{H}^{k_1,2}})
\\ +&CM\|{u(t)}\|_{\dot{H}^{3,2}}(M+\|{u(t)}\|_{\dot{H}^{k_2,2}}).
\endaligned
\end{equation}

By $(\ref{AgLa})$, we have
\begin{equation}\aligned \label{agla3}\|{u(t)}\|_{\dot{H}^{2,2}}&=\|\nabla {u(t)}\|_{\dot{H}^{1,2}}
\leq C\|{u(t)}\|_{\dot{H}^{1,2}}^{1/2}\|{u(t)}\|_{\dot{H}^{3,2}}^{1/2}
\\ \|{u(t)}\|_{\dot{H}^{k_1,2}}&=\|\nabla {u(t)}\|_{\dot{H}^{k_1-1,2}} \leq
C\|{u(t)}\|^{1-(k_1-1)/2}_{\dot{H}^{1,2}} \|{u(t)}\|^{(k_1-1)/2}_{\dot{H}^{3,2}}
\\ \|{u(t)}\|_{\dot{H}^{k_2,2}}&=\|\nabla {u(t)}\|_{\dot{H}^{k_2-1,2}} \leq
C\|{u(t)}\|^{1-(k_2-1)/2}_{\dot{H}^{1,2}} \|{u(t)}\|^{(k_2-1)/2}_{\dot{H}^{3,2}}.
\endaligned
\end{equation}

Applying $(\ref{agla3})$ to $(\ref{glob1c})$ and recalling that we have assumed both $\|{u(t)}\|_{\dot{H}^{2,2}}$ and $\|{u(t)}\|_{\dot{H}^{3,2}}$ are no less than $1$, we have
\begin{equation}\label{glob1b}\aligned |K_1|\leq &C(M)(\|{u(t)}\|_{\dot{H}^{3,2}}^{3/2}+ \|{u(t)}\|_{\dot{H}^{3,2}}^{1+k_1/2}+\|{u(t)}\|_{\dot{H}^{3,2}}^{(k_2+1)/2}),
\endaligned
\end{equation}
where $C(M)$ indicates that $C$ is a function only of $M$.  Choosing $\varepsilon=1/4$, we get
\begin{equation}\aligned \label{glob1a}|K_1|\leq
&C(M)\|{u(t)}\|^{15/8}_{\dot{H}^{3,2}}
\endaligned
\end{equation}
which finishes our $K_1$ estimate.  

For $K_2$, using Holder's inequality, we have
\begin{equation}\aligned \label{glob1er}|K_2|&\leq \|{u(t)}\|_{\dot{H}^{2,2}}\|A^{1/2}(\nabla {u(t)}\nabla {u(t)})\|_{L^2}.
\\ &\leq C\|{u(t)}\|^2_{\dot{H}^{2,2}}\|\nabla {u(t)}\|_{L^\infty}.
\endaligned
\end{equation}
Using $(\ref{agla2})$ and $(\ref{agla3})$ gives
\begin{equation}\label{glob1b2}\aligned |K_2|&\leq C\|{u(t)}\|_{\dot{H}^2}^2(M+\|{u(t)}\|_{\dot{H}^{k_2,2}})
\\ &\leq CM^2\|{u(t)}\|_{\dot{H}^{3,2}}+ CM^{23/8}\|{u(t)}\|_{\dot{H}^{3,2}}^{15/8}
\leq C(M)\|{u(t)}\|_{\dot{H}^{3,2}}^{15/8},
\endaligned
\end{equation}
and this finishes our work on $K_2$.  Combining the estimate for $K_1$ (equation $(\ref{glob1a}))$ and the estimate for $K_2$ (equation $(\ref{glob1b2}))$, we bound $I$ by 
\begin{equation}\label{aldo16}\aligned |I|\leq &C(M)\|{u(t)}\|^{15/8}_{\dot{H}^3}.
\endaligned
\end{equation}
Applying Young's multiplicative inequality with $q=16/15$,  we get
\begin{equation}\label{aldo4}I\leq \varepsilon(\|{u(t)}\|_{\dot{H}^3}^{15/8})^{16/15}
+C(M)(\varepsilon)^{-1}.
\end{equation}

Choosing $\varepsilon=1/4$, our final bound for $I$ is 
\begin{equation}\label{mery1}I\leq \frac{1}{4}\|{u(t)}\|_{\dot{H}^3}^2
+C(M).
\end{equation}

Now we turn our attention to $J$.

\subsubsection{Estimating $J$}\label{Estimating J}

As in the preceding subsection, we begin by writing  $J$ as $J=L_1+L_2$, where 
\begin{equation*}\aligned L_1&=-(AP^\alpha(\div({u(t)}\tensor {v(t)})),A{u(t)}),
\\ L_2&=-(AP^\alpha\div(1-\alpha^2\triangle)^{-1}(\nabla {u(t)}\nabla {v(t)}),A{u(t)}).
\endaligned
\end{equation*}

Starting with $L_1$, making liberal use of integration by parts, the product rule, and Holder's inequality, we have
\begin{equation}\label{noire2}\aligned |L_1|&\leq (A^{1/2}(\div({u(t)}\tensor {v(t)})),A^{3/2}{u(t)})
\\ &\leq C\|A^{3/2}{u(t)}\|_{L^2}(\|{v(t)}(A{u(t)})\|_{L^2}+\|(A^{1/2}{v(t)}) (A^{1/2}{u(t)})\|_{L^2}+\|(A{v(t)}){u(t)}\|_{L^2})
\\ &\leq C\|{u(t)}\|_{\dot{H}^{3,2}}(\|{v(t)}\|_{L^\infty}\|A{u(t)}\|_{L^2}+\|A^{1/2}{u(t)}\|_{L^2}\|A^{1/2} {v(t)}\|_{L^\infty}+\|{u(t)}\|_{L^{\tilde{p}}}\|A{v(t)}\|_{L^p})
\\ &\leq C\|{u(t)}\|_{\dot{H}^{3,2}}(\|{v(t)}\|_{L^\infty}\|A{u(t)}\|_{L^2}+\|A^{1/2}{u(t)}\|_{L^2}\|A^{1/2}{v(t)}\|_{L^\infty} +\|A^{1/2}{u(t)}\|_{L^2}\|A{v(t)}\|_{L^p}),
\endaligned
\end{equation}
where $\tilde{p}$ is as in equation $(\ref{rubio1})$.  By Sobolev embedding (and recalling that $p>3$), we have 
\begin{equation*}\aligned \|{v(t)}\|_{L^\infty}+\|A^{1/2}{v(t)}\|_{L^\infty}+\|A{v(t)}\|_{L^p}\leq \|{v(t)}\|_{H^{2,p}}<N,
\endaligned
\end{equation*}
where the last inequality is due to equation $(\ref{uniform v bound})$.  Applying this to equation $(\ref{noire2})$, we get 
\begin{equation*}|L_1|\leq C(N)\|{u(t)}\|_{\dot{H}^{3,2}}(\|A{u(t)}\|_{L^2}+\|A^{1/2}{u(t)}\|_{L^2}).
\end{equation*}

As in the estimate for $K_1$, we use $(\ref{agla3})$ to get  
\begin{equation}\label{glob1ca}\aligned |L_1|&\leq C(N)\|{u(t)}\|_{\dot{H}^{3,2}}(\|{u(t)}\|_{\dot{H}^{1,2}}^{1/2}\|{u(t)}\|^{1/2}_{\dot{H}^{3,2}}+\|{u(t)}\|_{\dot{H}^{1,2}})
\\ &\leq C(N,M)\|{u(t)}\|_{\dot{H}^{3,2}}^{3/2},
\endaligned
\end{equation}
where we recall that $\|{u(t)}\|_{H^{1,2}}\leq M$.  This finishes our $L_1$ estimate.  For $L_2$, using Holder's inequality, we have
\begin{equation*}\aligned \label{glob1era}|L_2|&\leq \|{u(t)}\|_{\dot{H}^{2,2}}\|A^{1/2}(\nabla {u(t)}\nabla {v(t)})\|_{L^2}.
\\ &\leq C\|{u(t)}\|^2_{\dot{H}^{2,2}}\|\nabla {v(t)}\|_{L^\infty}\leq C(N)\|{u(t)}\|_{\dot{H}^{2,2}}^2.
\endaligned
\end{equation*}
Using $(\ref{agla3})$ gives
\begin{equation}\label{glob1b234}\aligned |L_2|&\leq C(N)\|{u(t)}\|_{\dot{H}^{3,2}}\|{u(t)}\|_{\dot{H}^{1,2}}\leq C(N,M)\|{u(t)}\|_{\dot{H}^{3,2}},
\endaligned
\end{equation}
and this finishes our work on $L_2$.  Combining equations $(\ref{glob1ca})$ and $(\ref{glob1b234})$, we bound $J$ by 
\begin{equation}\label{rahan}|J|\leq C(N,M)\|{u(t)}\|_{\dot{H}^{3,2}}^{3/2}.
\end{equation}

Applying Young's inequality for products (and choosing $\varepsilon=1/4$), we get 
\begin{equation}\label{mery2}|J|\leq \frac{1}{4}(\|{u(t)}\|_{\dot{H}^{3,2}}^{3/2})^{4/3}+C(N,M)=\frac{1}{4}\|{u(t)}\|_{\dot{H}^{3,2}}^2+C(N,M).
\end{equation} 

We conclude this section by combining equations $(\ref{mery1})$ and $(\ref{mery2})$ to get  
\begin{equation}\label{noire3}|I|+|J|\leq \frac{1}{2}\|{u(t)}\|_{\dot{H}^{3,2}}^2+C(N,M).
\end{equation}

\subsubsection{Prove of equation $(\ref{noire1})$}
Returning to equation $(\ref{aldo})$, and using $(\ref{aldo1})$, $(\ref{aldo2})$, and $(\ref{noire3})$, we get 
\begin{equation*}\partial_t \|{u(t)}\|_{\dot{H}^{2,2}}^2+\|{u(t)}\|_{\dot{H}^{3,2}}^2\leq  \frac{1}{2}\|{u(t)}\|_{\dot{H}^{3,2}}^2+C(N,M).
\end{equation*}

Subtracting $\|{u(t)}\|_{\dot{H}^{3,2}}^2$ from both sides, we finally get 
\begin{equation*}\partial_t \|{u(t)}\|_{\dot{H}^{2,2}}^2\leq  -\frac{1}{2}\|{u(t)}\|_{\dot{H}^{3,2}}^2+C(N,M)\leq C(N,M).
\end{equation*}
Integrating from $a$ to $t$, we get 
\begin{equation*}\|u(t)\|_{\dot{H}^{2,2}}^2\leq \|u(a)\|_{\dot{H}^{2,2}}^2+\int_a^{T^*} C(N,M)ds\leq K.
\end{equation*} 
Taking the supremum over $t\in [a,T^*)$ gives equation $(\ref{noire1})$ and completes the proof of Lemma \ref{mlansglobalthm2}.

\subsection{An abstract interpolation result}\label{an abstract interpolation result}
In this subsection we address the following result.
\begin{lemma}Let $v$ be as in Theorem $\ref{fozzy1}$, and let $u$ be the global solution to the mLANS equation given by Theorem $\ref{fozzy1}$.  Also assume $u_0+v_0\in [B^{3/p}_{p,q}(\mathbb{R}^3), B^{3/2}_{2,q}(\mathbb{R}^3)]_{\theta,q}$ for some $\theta\in (0,1)$.  Then, for all $t$, $u(t)+v(t)\in [B^{3/p}_{p,q}(\mathbb{R}^3), B^{3/2}_{2,q}(\mathbb{R}^3)]_{\theta,q}$.
\end{lemma}
This is a specific case of the result proven in Section $4.4$ in \cite{galplan}.  As stated in equations $(4.11)$ and $(4.12)$ there, the two key requirements, adapted to this case, are that 
\begin{equation*}B^{3/p}_{p,q}(\mathbb{R}^3)\hookrightarrow [B^{3/p}_{p,q}(\mathbb{R}^3), B^{3/2}_{2,q}(\mathbb{R}^3)]_{\theta,q}\hookrightarrow B^{3/2}_{2,q}(\mathbb{R}^3),
\end{equation*}
and that 
\begin{equation*}\aligned \|v(t)\|_{B^{3/p}_{p,q}}&\leq C\|v_0\|_{B^{3/p}_{p,q}}
\\ \|u(t)\|_{B^{3/2}_{2,q}}&\leq C(\|v_0\|_{B^{3/p}_{p,q}})\|u_0\|_{B^{3/2}_{2,q}}.
\endaligned
\end{equation*}
The first requirement follows directly from the Besov embeddings in equation $(\ref{besov embedding})$.  For the second requirement, the first part follows from Theorem $\ref{old local theorem}$.  The second part follows from the fact that, by Theorem $\ref{fozzy1}$, $u\in BC([0,T):B^{3/2}_{2,q}(\mathbb{R}^3))$ for any $T>0$.  This result, combined with Theorem $\ref{fozzy1}$, completes the proof of Theorem $\ref{main theorem}$.

\section{Higher regularity for the local existence result}\label{Extending the local result: Higher regularity}
Here we prove Corollary $\ref{cor34}$.  The proof is an induction argument, similar to the one in \cite{sobpaper} applied to the LANS equation (which was in turn inspired by the argument in \cite{katoinduction} for the Navier-Stokes equation).

As usual, before stating the theorem, we construct a solution to the LANS equation $v$.  Here, we pick $p>n$, and let $v_0\in B^{n/p}_{p,q}(\mathbb{R}^n)$ with $\|v_0\|_{B^{n/p}_{p,q}}$ arbitrarily small, so by Theorem $\ref{old local theorem}$ and Corollary $\ref{old local extension}$, we have a global solution $v$ to the LANS equation where $v\in BC([0,T):B^{n/p}_{p,q}(\mathbb{R}^n)\cap \dot{C}^T_{a;r,p,q})$, with $a=(r-n/p)/2$ for any real $r>n/p$.  
\begin{lemma}\label{higher reg}With $v$ as in the preceding paragraph, let $u_0\in B^{n/2}_{2,q}(\mathbb{R}^n)$ and let $u$ be the associated unique solution to the mLANS equation with initial data $u_0$ such that 
\begin{equation*}u\in BC([0,T):B^{n/2}_{2,q}(\mathbb{R}^n))\cap \dot{C}^T_{({s}-n/p)/2;{s},2,q},
\end{equation*}
where $0<{s}-n/2<1$.  Then for all $k\geq s$, we have that $u\in \dot{C}^T_{(k-n/2)/2;k,2,q}$.
\end{lemma}

\begin{proof}
We start with the solution to the mLANS equation $u$.  Then let $\delta>0$ be arbitrary, and let $w=t^\delta u$.  We note that $w(0)=0$.  Then 
\begin{equation*}\aligned \partial_t w&=\delta t^{\delta-1} u+t^\delta \partial_t u
\\ &=\delta t^{-1} w+t^\delta (\triangle u-\div(u\tensor u+\tau^\alpha (u,u))-\div(u\tensor v+\tau^\alpha(u,v)))
\\ &=\delta t^{-1} w+ \triangle w-t^{-\delta}\div (w\tensor w+\tau^\alpha (w,w))-\div(w\tensor v+\tau^\alpha(w,v)).
\endaligned
\end{equation*}
Applying Duhamel's principle, we get 
\begin{equation*}\aligned w=e^{t\triangle}w_0+&\int_0^t e^{(t-s)\triangle}s^{-1}w(s)ds+\int_0^t e^{(t-s)\triangle}s^{-\delta}(\div(w(s)\tensor w(s)+\tau^\alpha(w(s),w(s))))ds
\\ + &\int_0^t e^{(t-s)\triangle}(\div(w(s)\tensor v(s)+\tau^\alpha(w(s),v(s))))ds.
\endaligned
\end{equation*}
Recalling that $w(0)=w_0=0$, and substituting $w=t^\delta u$, we get 
\begin{equation*}\aligned u=&t^{-\delta}\int_0^t e^{(t-s)\triangle}s^{\delta-1}u(s)ds+t^{-\delta}\int_0^t e^{(t-s)\triangle}s^\delta(\div(u(s)\tensor u(s)+\tau^\alpha(u(s),u(s)))) ds
\\ +&t^{-\delta}\int_0^t e^{(t-s)\triangle}s^\delta(\div(u(s)\tensor v(s)+\tau^\alpha(u(s),v(s)))) ds.
\endaligned
\end{equation*}

Now we are ready to apply the induction.  We have by assumption that $u$ is in $\dot{C}^T_{(s-n/2)/2;s,2,q}$, where $s> 1$.  For induction, we assume this solution $u$ is also in $\dot{C}^T_{(k-n/2)/2;k,2,q}$, and seek to show that $u$ is in $\dot{C}^T_{(k+h-n/2)/2;k+h,2,q}$, where $0<h<1$ is fixed and will be chosen later.  We have  
\begin{equation*}\aligned \|u\|_{B^{k+h}_{2,q}}&\leq I+J_1+J_2+K_1+K_2,
\endaligned
\end{equation*}
with $I$, $J_1$, $J_2$, $K_1$, and $K_2$ defined by  
\begin{equation*}\aligned I&=t^{-\delta}\int_0^t \|e^{(t-s)\triangle} s^{\delta-1}u(s)\|_{B^{k+h}_{2,q}}ds 
\\ J_1&=t^{-\delta}\int_0^t \|e^{(t-s)\delta} s^{\delta} (\div(1-\alpha^2\triangle)^{-1}(\nabla u(s)\nabla u(s)))\|_{B^{k+h}_{2,q}}ds       
\\ J_2&=t^{-\delta}\int_0^t \|e^{(t-s)\delta} s^{\delta}(\div(u(s)\tensor u(s)))\|_{B^{k+h}_{2,q}}ds
\\ K_1&= t^{-\delta}\int_0^t \|e^{(t-s)\delta} s^{\delta} (\div(1-\alpha^2\triangle)^{-1}(\nabla u(s)\nabla v(s)))\|_{B^{k+h}_{2,q}}ds    
\\ K_2&=t^{-\delta}\int_0^t \|e^{(t-s)\delta} s^{\delta}(\div(u(s)\tensor v(s)))\|_{B^{k+h}_{2,q}}ds,
\endaligned
\end{equation*}
where, as usual, we have suppressed terms from $\tau^\alpha$ that are controlled by the terms we included.  The $I$, $J_1$, and $J_2$ terms are the terms from the LANS equation, while $K_1$ and $K_2$ are the terms resulting from the modification of the LANS equation.  We address $I$, $J_1$, and $J_2$ first. 

\subsection{Bounding $I$, $J_1$, and $J_2$}
Starting with $I$, we have  
\begin{equation}\label{fsu1}\aligned I&\leq t^{-\delta}\int_0^t |t-s|^{-h/2}s^{\delta-1}\|u(s)\|_{B^{k}_{2,q}}
\\ &\leq t^{-\delta}\|u\|_{(k-{n/2})/2;k,2,q}\int_0^t |t-s|^{-h/2}s^{\delta-1-(k-{n/2})/2}ds
\\ &\leq C\|u\|_{(k-{n/2})/2;k,2,q}t^{-\delta}t^{-h/2}t^{\delta-1-(k-{n/2})/2+1}
\\ &\leq Ct^{-(k+h-{n/2})/2}\|u\|_{(k-{n/2})/2;k,2,q},
\endaligned
\end{equation}
provided 
\begin{equation*}\aligned 1&>h/2,
\\ -1&<\delta-1-(k-{n/2})/2,
\endaligned
\end{equation*}
which clearly holds for sufficiently large $\delta$.  We observe that, without modifying the PDE to include these $t^\delta$ terms, we would need $(k-{n/2})/2$ to be less than $1$, which does not hold for large $k$.

For $J_1$, we choose $\tilde{r}=n/2-1-\varepsilon$, and with $n/\tilde{p}=n/2+\tilde{r}$, we have 
\begin{equation}\label{fsu2}\aligned J_1\leq &t^{-\delta}\int_0^t|t-s|^{-(h+n/\tilde{p}-n/2)/2}s^{\delta} \|\div(1-\triangle)^{-1}(\nabla u \nabla u)\|_{B^{k}_{\tilde{p},q}}ds 
\\ \leq &t^{-\delta}\int_0^t|t-s|^{-(h-n/2+\tilde{r})/2}s^{\delta} \|(\nabla u \nabla u)\|_{B^{k-1}_{\tilde{p},q}}ds 
\\ \leq &t^{-\delta}\int_0^t|t-s|^{-(h-n/2+\tilde{r})/2}s^{\delta} \|\nabla u \|_{B^{k-1}_{2,q}}\|\nabla u\|_{B^{\tilde{r}}_{2,q}}ds 
\\ \leq &t^{-\delta}\int_0^t|t-s|^{-(h-n/2+\tilde{r})/2}s^{\delta} \|u\|_{B^{k}_{2,q}}\|u\|_{B^{n/2-\varepsilon}_{2,q}}ds 
\\ \leq &t^{-\delta}\|u\|_{(k-{n/2})/2;k,2,q}\|u\|_{0;n/2,2,q} \int_0^t|t-s|^{-(h+n/2-\tilde{r})/2}s^{\delta-(k-{n/2})/2}ds 
\\ \leq &t^{-(h+k-\tilde{r})/2+1}\|u\|_{(k-{n/2})/2;k,2,q}\|u\|_{0;n/2,2,q}
\\ \leq &t^{-(h+k-n/2)/2+(1-\varepsilon)/2}\|u\|_{(k-{n/2})/2;k,2,q}\|u\|_{0;n/2,2,q},
\endaligned
\end{equation}
provided 
\begin{equation*}\aligned \delta&>(k-{n/2})/2+(1-{n/2})/2,
\\ 2&>h+{n/2},
\endaligned
\end{equation*}  
and we again see that this is easily satisfied by choosing $\delta$ large and $h$ small.  For $J_2$, we define $\tilde{s}=n/2-\varepsilon$ and $n/\tilde{q}=n/2+\tilde{s}$, and have 
\begin{equation}\label{fsu3}\aligned J_2\leq &t^{-\delta}\int_0^t|t-s|^{-(h+1+n/\tilde{q}-{n/2})/2}s^{\delta} \|u\tensor u\|_{B^{k}_{\tilde{q},q}}ds 
\\ \leq &t^{-\delta}\int_0^t|t-s|^{-(h+1+n/2-\tilde{s})/2}s^{\delta} \|u\|_{B^{k}_{2,q}}\|u\|_{B^{n/2}_{2,q}}ds 
\\ \leq&t^{-\delta}\|u\|_{(k-{n/2})/2;k,2,q}\|u\|_{0;{n/2},2,q} \int_0^t|t-s|^{-(h+1-\varepsilon)/2}s^{\delta-(k-{n/2})/2}ds 
\\ \leq &t^{-(h+k-n/2)/2+1/2+\varepsilon}\|u\|_{(k-{n/2})/2;k,2,q}\|u\|_{0;{n/2},2,q},
\endaligned
\end{equation}
provided 
\begin{equation*}\aligned 1&>h-\varepsilon, 
\\ -1&<\delta-(k-{n/2})/2.
\endaligned
\end{equation*}
Combining equations $(\ref{fsu1})$, $(\ref{fsu2})$ and $(\ref{fsu3})$, we have that, for $h$ small enough and $\delta$ large enough, 
\begin{equation}\label{fsu4}I+J_1+J_2\leq Ct^{-(h+k-n/2)/2}\|u\|_{(k-{n/2})/2;k,2,q}\|u\|_{0;{n/2},2,q}.
\end{equation}
Now we turn our attention to $K_1$ and $K_2$.

\subsection{Bounding $K_1$ and $K_2$}

Starting with $K_1$, Defining $n/\tilde{p}=n/p+n/2$ and $a=(r-n/p)/2$, we have 
\begin{equation}\label{jet1}\aligned &t^{-\delta}\int_0^t|t-s|^{-(h+n/\tilde{p}-n/2)/2}s^{\delta} \|\div(1-\triangle)^{-1}(\nabla u \nabla v)\|_{B^{k}_{\tilde{p},q}}ds 
\\ \leq &t^{-\delta}\int_0^t|t-s|^{-(h+n/p)/2}s^{\delta} \|(\nabla u \nabla v)\|_{B^{k-1}_{\tilde{p},q}}ds 
\\ \leq &t^{-\delta}\int_0^t|t-s|^{-(h+n/p)/2}s^{\delta} \|\nabla u \|_{B^{k-1}_{2,q}}\|\nabla v\|_{L^p}ds 
\\ \leq &t^{-\delta}\int_0^t|t-s|^{-(h+n/p)/2}s^{\delta} \| u \|_{B^{k}_{2,q}}\|v\|_{B^{r}_{p,q}}ds 
\\ \leq &t^{-\delta}\|u\|_{(k-n/2)/2;k,2,q}\|v\|_{a;r,p,q} \int_0^t|t-s|^{-(h+n/p)/2}s^{\delta-(k-n/2)/2-a}ds 
\\ \leq &t^{-(h+k-n/2)/2+1-r/2}\|u\|_{(k-n/2)/2;k,2,q}\|v\|_{a;r,p,q},
\endaligned
\end{equation}
provided 
\begin{equation*}\aligned \delta&>(k-n/2)/2+(1-n/p)/2,
\\ r&<2,
\\ 2&>h+n/p,
\endaligned
\end{equation*}
all of which are easily satisfied by a sufficiently large choice of $\delta$ and a sufficiently small choice of $h$.  

For $K_2$, we have  
\begin{equation}\label{jet2}\aligned &t^{-\delta}\int_0^t|t-s|^{-(h+1+n/\tilde{p}-n/2)/2}s^{\delta} \|u\tensor u\|_{B^{k}_{\tilde{p},q}}ds 
\\ \leq &t^{-\delta}\int_0^t|t-s|^{-(h+1+n/p)/2}s^{\delta} \|u\|_{B^{k}_{2,q}}\|v\|_{L^p}ds 
\\ \leq&t^{-\delta}\|u\|_{(k-n/2)/2;k,2,q}\|v\|_{0;n/p,p,q} \int_0^t|t-s|^{-(h+1+n/p)/2}s^{\delta-(k-n/2)/2}ds 
\\ \leq &t^{-(h+k-n/2)/2+(1-n/p)/2}\|u\|_{(k-n/2)/2;k,2,q}\|v\|_{0;n/p,p,q},
\endaligned
\end{equation}
provided 
\begin{equation*}\aligned 1&>h+n/p,
\\ \delta&>(k-n/2)/2.
\endaligned
\end{equation*}  
So we have that 
\begin{equation}\label{fsu5}K_1+K_2\leq t^{-(h+k-n/2)/2}\|u\|_{(k-n/2)/2;k,2,q}\|v\|_{(r-n/p)/2;r,p,q}.
\end{equation}

\subsection{Finishing the proof of Lemma $\ref{higher reg}$}

Using equations $(\ref{fsu4})$ and $(\ref{fsu5})$, we get 
\begin{equation*}\aligned \|u\|_{H^{k+h,p}}&\leq Ct^{(k+h-n/2)/2}\|u\|_{(k-n/2)/2;k,2,q}(|u\|_{0;{n/2},2,q}+ \|v\|_{(r-n/p)/2;r,p,q})
\endaligned
\end{equation*}
which immediately gives
\begin{equation*}\aligned \|u\|_{(k+h-{n/2})/2;k+h,p}\leq C\|u\|_{(k-n/2)/2;k,2,q}(|u\|_{0;{n/2},2,q}+ \|v\|_{(r-n/p)/2;r,p,q}),
\endaligned
\end{equation*}
which proves the desired result.  We remark that $\delta$ is chosen after beginning the induction step, while the appropriate value of $h$ is fixed by the choices of $n,$ $p,$ and ${n/2}$.
\end{proof}

\section{Appendix: A Modified Product Estimate}\label{A Modified Product Estimate}
In this appendix we prove Proposition $\ref{new product estimate}$, which can be found in Corollary $1.3.1$ in \cite{chemin}.  Before beginning, we establish another result for the Littlewood-Paley operators and make a slight notational change.  First, we observe that, by changing variables, 
\begin{equation}\label{omega results}\|\psi_j\|_{L^p}\leq 2^{jn/p'}\|\psi_0\|_{L^p}\leq C2^{jn/p'},
\end{equation}
where $p'$ is the Holder' conjugate to $p$, i.e. $1=1/p+1/p'$.

Next, we make a slight notational change.  For $j>0$, we leave $\psi_j$ as defined in Section $\ref{Besov spaces}$.  For $j=0$, we set $\psi_0=\Psi$, so $\hat{\psi_0}$ is now supported on the ball centered at the origin of radius $1/2$ and $\triangle_0 f=\psi_0 *f=\Psi *f$.  Then the Besov norm can be defined by 
\begin{equation*}\|f\|_{B^r_{p,q}}=\(\sum_{j=0}^\infty 2^{rjq}\|\triangle_j u\|_{L^p}^q\)^{1/q}.
\end{equation*}
We are now ready to prove Proposition $\ref{new product estimate}$.

\begin{proof}[Proposition $\ref{new product estimate}$]We start by taking the $L^p$ norm of equation $(\ref{bony256})$, and get:
\begin{equation*}\aligned \|\triangle_j (fg)\|_{L^p}\leq &\sum_{k=-3}^3 \|\triangle_j (S_{j+k-3}f\triangle_{j+k} g)\|_{L^p}+ \sum_{k=-3}^3 \|\triangle_j (S_{j+k-3}g\triangle_{j+k} f)\|_{L^p}
\\ +&\sum_{k>j-4}\|\triangle_j \(\triangle_k f\sum_{l=-2}^2 \triangle_{k+l}g\)\|_{L^p}.
\endaligned
\end{equation*}
We first observe that, without loss of generality, we can set $k=l=0$ in the finite sums and replace $k>j-4$ with $k>j$.  Doing so, we get 
\begin{equation*}\aligned \|\triangle_j (fg)\|_{L^p}\leq &\|\triangle_j (S_{j-3}f\triangle_{j} g)\|_{L^p}+ \|\triangle_j (S_{j-3}g\triangle_{j} f)\|_{L^p}
\\ +&\sum_{k>j}\|\triangle_j \(\triangle_k f\triangle_{k}g\)\|_{L^p}.
\endaligned
\end{equation*}

Starting with the first term, and defining $\tilde{p}$ by $1+1/p=1/\tilde{p}+1/p_2$, we have 
\begin{equation*}\aligned \|\triangle_j (S_{j-3}f\triangle_{j} g)\|_{L^p}&\leq \|\psi_j\|_{L^{\tilde{p}}}\|\triangle_j f S_{j-3} g\|_{L^{p_2}} 
\leq C2^{jn/\tilde{p}'}\|\triangle_j g\|_{L^{p_2}}\|S_{j-3}f\|_{L^\infty} 
\\ &\leq C2^{jn/\tilde{p}'} \|\triangle_j g\|_{L^{p_2}}\sum_{m<j-3}\|\triangle_m f\|_{L^\infty} 
\\ &\leq C2^{jn(1/p_2-1/p)/\tilde{p}'} \|\triangle_j g\|_{L^{p_2}}\sum_{m<j-3}2^{mn/p_1}\|\triangle_m f\|_{L^{p_1}},
\endaligned
\end{equation*}
where we used Young's inequality, equation $(\ref{omega results})$, Holder's inequality, and finally Bernstein's inequality.  

A similar calculation for the second term yields
\begin{equation*}\|\triangle_j (S_{j-3}g\triangle_{j} f)\|_{L^p}\leq C2^{jn(1/p_1-1/p)} \|\triangle_j f\|_{L^{p_2}}\sum_{m<j-3}2^{mn/p_2}\|\triangle_m g\|_{L^{p_1}}.
\end{equation*}

For the third term, we have 
\begin{equation*}\aligned \sum_{k>j} \|\triangle_j(\triangle_k f \triangle_k g\|_p)
&\leq \|\psi_j \|_{\tilde{q}}\sum_{k>j}\|\triangle_k u \triangle_k v\|_{L^q}
\\ &\leq 2^{jn/\tilde{p}'}\sum_{k>j}\|\triangle_k f\|_{p_1}\|\triangle_k g\|_{p_2}
\\ &\leq 2^{jn(1/p-1/p_1-1/p_2)}\sum_{k>j} \|\triangle_k f\|_{p_1}\|\triangle_k g\|_{p_2},
\endaligned
\end{equation*}
where $1+1/p=1/\tilde{q}+1/q$ and $1/q=1/p_1+1/p_2$. 

So we have that 
\begin{equation}\label{chuck1}\aligned &\|\triangle_j (fg)\|_{L^p}\leq 2^{jn(1/p_2-1/p)} \|\triangle_j g\|_{L^{p_2}}\sum_{m<j-3}2^{jn/p_1}\|\triangle_m f\|_{L^{p_1}}
\\ +&2^{jn(1/p_1-1/p)} \|\triangle_j f\|_{L^{p_1}}\sum_{m<j-3}2^{jn/p_2}\|\triangle_m g\|_{L^{p_2}}
\\ +&2^{jn(1/p-1/p_1-1/p_2)}\sum_{k>j} \|\triangle_k f\|_{p_1}\|\triangle_k g\|_{p_2}
\endaligned
\end{equation}

Multiplying $(\ref{chuck1})$ by $2^{j(s_1+s_2-n(1/p_2+1/p_1-1/p))}$ and taking the $l^q$ norm in $j$, we get 
\begin{equation*}\|fg\|_{B^s_{p,q}}\leq I+J+K,
\end{equation*}
where 
\begin{equation*}\aligned I&=\(\sum_j 2^{(s_1+s_2-n/p_1)jq} \|\triangle_j g\|_{L^{p_2}}^q(\sum_{m<j-3}2^{mn/p_1}\|\triangle_m f\|_{L^{p_1}})^q\)^{1/q},
\\ J&=\(\sum_j 2^{(s_1+s_2-n/p_2)jq} \|\triangle_j f\|^q_{L^{p_1}}(\sum_{m<j-3}2^{mn/p_2}\|\triangle_m g\|_{L^{p_2}})^q\)^{1/q},
\\ K&=\(\sum_j (2^{j(s_1+s_2)}\sum_{k>j} \|\triangle_k f\|_{p_1}\|\triangle_k g\|_{p_2})^q\)^{1/q}.
\endaligned
\end{equation*}

For $I$, we have 
\begin{equation*}\aligned I&\leq \(\sum_j 2^{(s_1+s_2-n/p_1)jq} \|\triangle_j g\|_{L^{p_2}}^q(\sum_{m<j-3}2^{jn/p_1}\|\triangle_m f\|_{L^{p_1}})^q\)^{1/q}
\\ &\leq \(\sum_j (2^{js_2}\|\triangle_j g\|_{L^{p_2}})^q(\sum_{m<j-3}2^{m(n/p_1+s_1-n/p_1)}2^{(j-m)(s_1-n/p_1)}\|\triangle_m f\|_{L^{p_1}})^q\)^{1/q}
\\ &\leq \|f\|_{B^{s_1}_{p_1,\infty}}\sum_k 2^{-(s_1-n/p_2)}\(\sum_j (2^{js_2}\|\triangle_j g\|_{L^{p_2}})\)^{1/q}
\\ &\leq \|f\|_{B^{s_1}_{p,q}}\|g\|_{B^{s_2}_{s_2,q}},
\endaligned
\end{equation*}
provided $s_1<n/p_1$.  A similar calculation for $J$ yields 
\begin{equation*}\aligned J&\leq \|f\|_{B^{s_1}_{p,q}}\|g\|_{B^{s_2}_{s_2,q}},
\endaligned
\end{equation*}
provided $s_2<n/p_2$.  For $K$, we have, using Young's inequality for sums,  
\begin{equation*}\aligned K&=\(\sum_j (\sum_{k>j}2^{(j-k)(s_1+s_2)} 2^{ks_1}\|\triangle_k f\|_{p_1}2^{ks_2}\|\triangle_k g\|_{p_2})^q\)^{1/q}
\\ &\leq \|g\|_{B^{s_2}_{p_2,\infty}}\(\sum_j (\sum_{k>j}2^{(j-k)(s_1+s_2)} 2^{ks_1}\|\triangle_k f\|_{p_1})^q\)^{1/q}
\\ &\leq \|g\|_{B^{s_2}_{p_2,\infty}}\sum_{k}2^{-k(s_1+s_2)}\(\sum_k 2^{ks_1}\|\triangle_k f\|_{p_1})^q\)^{1/q}
\\ &\leq C\|f\|_{B^{s_1}_{p_1,q}}\|g\|+{B^{s_2}_{p_2,q}},
\endaligned
\end{equation*}
provided $s_1+s_2>0$.  This finishes the proposition.
\end{proof}

\bibliographystyle{amsplain}
\bibliography{references}

\end{document}